\documentclass[11pt]{article}
\usepackage{times}
\usepackage{amsmath, amsfonts, amssymb, amsthm, bm, stmaryrd}
\usepackage{authblk}
\renewcommand{\epsilon}{\varepsilon}
\renewcommand{\theta}{\vartheta}
\renewcommand{\phi}{\varphi}
\renewcommand{\rho}{\varrho}

\usepackage{graphicx}
\usepackage[font=small,labelfont=bf]{caption}

\newtheorem{Def}{Definition}[section]
\newenvironment{definition}{\begin{Def} \rm}{\end{Def}}
\newtheorem{lemma}[Def]{Lemma}
\newtheorem{proposition}[Def]{Proposition}

\newtheorem{theorem}[Def]{Theorem}

\newcommand{\AufzAnfang}{\begin{enumerate}}
\newcommand{\AufzEnde}{\end{enumerate}}
\newcommand{\Nummer}[1]{\item[{\rm (#1)}]\parindent0pt\parskip1ex}
\newcommand{\AxiomeAnfang}{\begin{itemize}\item[]\begin{itemize}}
\newcommand{\AxiomeEnde}{\end{itemize}\end{itemize}}
\newcommand{\Axiom}[1]{\item[{\rm (#1)}]\parindent0pt\parskip1ex}

\newcommand{\komma}{,\hspace{0.6em}}
\renewcommand{\emptyset}{\varnothing}
\renewcommand{\leq}{\leqslant}
\renewcommand{\geq}{\geqslant}

\newcommand{\Reals}{{\mathbb R}}
\newcommand{\ExtReals}{\bar{\mathbb R}^+}

\newcommand{\true}{\top}
\newcommand{\false}{\bot}
\newcommand{\impl}{\rightarrow}
\newcommand{\implc}[1]{\stackrel{#1}{\rightarrow}}

\newcommand{\comp}{\mathbin{\nearrow}}
\newcommand{\boolequ}{\mathbin{\approx}}
\newcommand{\Tleq}{\preccurlyeq}
\newcommand{\Tequ}{\mathbin{\approx}}
\newcommand{\Tequcl}[1]{\langle #1 \rangle}
\DeclareMathOperator{\compl}{\lnot}
\newcommand{\CPL}{$\mathsf{CPL}$}
\newcommand{\LAE}{$\mathsf{LAE}^q$}
\newcommand{\pLAE}{$\mathbf{LAE}$}
\newcommand{\LSE}{$\mathsf{LSE}$}
\newcommand{\proves}{\vdash}

\newcommand{\Add}[1]{\hspace{0.3em}\mbox{\footnotesize #1}}
\newcommand{\Rule}[1]{\text{(#1)} \hspace{1em}}

\newcommand{\widerspruch}{\divideontimes}

\begin{document}

\title{Logic of approximate entailment \\
in quasimetric spaces
\thanks{Preprint of an article published by Elsevier in the {\sl International Journal of Approximate Reasoning} {\bf 64} (2015), 39-53. It is available online at: {\tt
https://www.sciencedirect.com/science/article/pii/S0888613X15000924}.}}

\author{Thomas Vetterlein}

\affil{\footnotesize Department of Knowledge-Based Mathematical Systems \\
Johannes Kepler University Linz \\
Altenberger Stra\ss e 69, 4040 Linz, Austria \\
{\tt Thomas.Vetterlein@jku.at}}

\maketitle

\begin{abstract}\parindent0pt\parskip1ex

\mbox{}\vspace{-3ex}

The logic \LAE{} discussed in this paper is based on an approximate entailment relation. \LAE{} generalises classical propositional logic to the effect that conclusions can be drawn with a quantified imprecision. To this end, properties are modelled by subsets of a distance space and statements are of the form that one property implies another property within a certain limit of tolerance. We adopt the conceptual framework defined by E.\ Ruspini; our work is towards a contribution to the investigation of suitable logical calculi.

\LAE{} is based on the assumption that the distance function is a quasimetric. We provide a proof calculus for \LAE{} and we show its soundness and completeness for finite theories. As our main tool for showing completeness, we use a representation of proofs by means of weighted directed graphs.

\end{abstract}

\section{Introduction}
\label{sec:introduction}

Reasoning about a particular topic means exploiting the logical relationships between potential facts. For instance, if we know that some property $\alpha$ applies and that $\alpha$ implies another property $\beta$, we draw the conclusion that then also $\beta$ applies. Under practical circumstances, however, reasoning is often done in a somewhat flexible way. In fact, we typically draw the indicated conclusion even if the presence of $\alpha$ does not imply $\beta$ strictly. Knowing that whenever $\alpha$ holds we are in a situation close to a situation in which $\beta$ holds is often considered as a sufficient basis for an inference. Reasoning in such a way might be called ``approximate'' and its formalisation is the topic of the present paper.

Approximate reasoning can be realised in many different ways. For an overview of the many different approaches to this field, see, e.g., \cite{HAH}. Here, we adopt the conceptual framework that has been specified by E.\ Ruspini in his seminal paper \cite{Rus} in 1991.

Let us outline the main ideas. Assume that we are given a set of properties, suitable to describe some objects under consideration. Let a universe of discourse $W$, customarily called a set of worlds, be chosen, corresponding to the different possible states in which the objects can be. We may then assume that at each world $w$ an evaluation assigns to each property $\phi$ its truth value ``true'' or ``false''. Likewise, with each property $\phi$ we may associate the subset of $W$ consisting of those worlds at which $\phi$ holds. In this framework, the canonical way of reasoning obeys the rules of classical propositional logic (\CPL). In particular, implicational relationships correspond to subsethood: for two properties $\alpha$ and $\beta$ to validate $\alpha \impl \beta$ means that the set $A \subseteq W$ associated with $\alpha$ is contained in the set $B \subseteq W$ associated with $\beta$.

Practical reasoning, however, should be robust with regard to small changes. To allow for greater flexibility, Ruspini takes into account an additional aspect. The set of worlds describes the variety of states of the objects under consideration and there might be a natural way to specify the degree to which two such states resemble. A similarity relation on $W$ can be introduced for this purpose, mapping each pair of worlds $v, w \in W$ to an element $s(v, w)$ of the real unit interval, a larger value meaning a stronger resemblance. Such a mapping can be understood as a graded version of an accessibility relation in the sense of modal logic. In fact, for any given $r \in [0,1]$, we may define the binary relation $R_r$ on $W$ requiring $v \, R_r \, w$ if $s(v, w) \geq r$. Based on $R_r$, we may define a possibility operator $\Diamond_r$ on the power set of $W$. The expression $\alpha \impl \Diamond_r \beta$ then generalises the classical implication. At a world $w$ at which $\alpha$ is fulfilled, $\beta$ might not hold but still be regarded as an ``example'' of $\beta$. However, $w$ will be a ``typical example'' only if $r = 1$. For $r$ close to $1$, $w$ might still be considered as a ``good example'', but for small $r$, the world $w$ would be considered as a ``bad example''.

Ruspini's approach has been explored from a logical perspective in a series of papers by numerous authors. A comprehensive study is the paper \cite{EGGR}. The paper \cite{DPEGG} is a further early contribution and includes an application to interpolative reasoning as required in rule-based fuzzy control.  For a more recent article, which contains an overview of different approaches, see \cite{GoRo}.

A logical calculus may also be based on a graded analogue of the classical implication, without the detour via modal operators. In fact, Ruspini defines the {\it degree of implication} between properties $\alpha$ and $\beta$ modelled by subsets $A$ and $B$, respectively, by
\[ {\mathbf I}(B | A) = \inf_{v \in A} \sup_{w \in B} s(v,w). \]
We may hence define that $\alpha$ implies $\beta$ to the degree $r$ if ${\mathbf I}(B | A) \geq r$. This idea is exploited in the Ph.D.\ Thesis of R.\ Rodr\' iguez \cite{Rod} and is moreover the basis of the paper \cite{EGRV}.

The logic \LAE{} that we consider in this paper is conceptually in line with these last mentioned papers. Details differ, but our modifications do not affect the overall picture. Our syntactical objects are of the form
\begin{equation} \label{fml:LAE-formula}
\alpha \implc{d} \beta,
\end{equation}
where $\alpha$ and $\beta$ are formulas of \CPL{} and $d$ is a positive real number (possibly $0$). The two \CPL{} formulas are interpreted by subsets of a quasimetric space. A quasimetric is defined similarly to a metric, but the axiom of symmetry is dropped. The statement (\ref{fml:LAE-formula}) is defined to be satisfied if the set of worlds associated with $\alpha$ is contained in the $d$-neighbourhood of the set associated with $\beta$. We note that we work with a quasimetric and not with the dual notion of a quasisimilarity. For this reason the values associated with formulas are interpreted in the reverse way: the statement $\alpha \implc{d} \beta$ is the weaker, the larger the parameter $d$ is. In particular, $\alpha \implc{0} \beta$ holds whenever $\alpha \impl \beta$ is a tautology of \CPL.

\LAE{} could be regarded as a logic of quasimetric spaces. A different perspective on logics of metric spaces as well as logics of more general distance spaces has been developed in a further series of papers. In \cite{KSSWZ}, first-order logic on universes endowed with a distance function is studied. Related contributions include \cite{SSWZ}, \cite{Kut}, and \cite{STWZ}.

We may furthermore ask for the connections of our work to many-valued logic. We deal with implications between crisp properties, endowed with real values that express the degree to which the implicational relationships hold. The notion of a Graded Consequence Relation between sets of Boolean formulas and a further Boolean formula has been introduced by Chakraborty \cite{Cha,ChDu}. A central aim has been the adaptation of common logical concepts to a many-valued framework. In spite of the distinct scopes, Chakraborty's notion shares conceptually certain aspects with our framework: crisp properties are dealt with on the base level, a degree is added on the metalevel. An exact relationship has been established, however, only in a more general context. Namely, Graded Consequence Relations were generalised to hold between a fuzzy subset of formulas and a further formula; see, e.g., \cite{DuCh,REGG}. In \cite{REGG}, it is pointed out that Implicative Closure Operators capture the approximate entailment relation on which the present work is based.

The present work follows up the previous papers \cite{Rod,GoRo,EGRV}. To see in which respect our work represents a progress, let us go further into detail. The present approach differs in two important points from \cite{EGRV}. Firstly, the logic discussed in \cite{EGRV} is not based on a quasimetric but on a similarity relation. The relationship between a quasimetric and a similarity relation is an indirect one: a metric is a symmetric quasimetric and a similarity relation can be regarded as a generalisation of a metric. Secondly, in \cite{EGRV}, the language is required to consist of a fixed finite number of variables and each world is assumed to correspond to what is called a {\it maximal elementary conjunction}, or m.e.c.\ for short. A m.e.c.\ is a conjunction of literals in which all variables occur exactly once. Here, m.e.c.s are not used.

The latter aspect has in fact been the starting point for the present work: we no longer assume that the number of variables is bound by a fixed finite number. This change implies that we can not proceed in the same way as, e.g., in \cite{EGRV}. The axiomatisation of the logic cannot be done in the same style as before and the methods employed in completeness proofs are no longer available, as there is no counterpart of a m.e.c.\ here. In the proof system proposed in \cite{EGRV}, there are two axioms that rely on m.e.c.s. The first one expresses the symmetry of the distance function; it corresponds to the requirement that if a world is in the $d$-neighbourhood of another world, also the converse relation holds. A further axiom is concerned with disjunctions; it corresponds to the fact that if a world is in the neighbourhood of the union of two sets, then it is already in the neighbourhood of one of them.

A previous work whose aim has been to replace these two axioms is our contribution \cite{Vet}. In this case, we introduced an additional binary connective $\comp$. For two subsets $A$ and $B$ of a metric space, $A \comp B$ was supposed to contain all those elements whose distance to $A$ is at most as large as their distance to $B$. The same connective has also been considered in the framework of Logics of Comparative Similarity; see, e.g., \cite{AOS}. The presence of $\comp$ enabled us to overcome the problem of characterising the neighbourhood of the union of sets. The problem of characterising the symmetry of the distance function was not touched; symmetry was not assumed.

Here, we do not extend the language by additional connectives. Instead, we pose a seemingly natural question: Do we really need the axioms involving m.e.c.s for the axiomatisation of a logic of approximate entailment? Could it be the case that the remaining axioms are already sufficient?

We contribute a step towards an affirmation of the latter question. In fact, using quasimetrics as distance functions, we provide a positive answer. For the associated approximate entailment relation, we present a proof calculus that is restricted to the obviously necessary minimum of rules; nonetheless, we establish its soundness and completeness. We again have to include a finiteness condition though; we show the completeness of our calculus w.r.t.\ finite theories.

We proceed as follows. In the introductory Section \ref{sec:LAE}, we specify the logic \LAE{}, which deals with graded implications of the form (\ref{fml:LAE-formula}) and interprets properties by subsets of a quasimetric space. Section \ref{sec:Rules} introduces the calculus \pLAE, based on a small set of rules that are sound for \LAE. We furthermore establish that proofs in \pLAE{} can be represented in the form of weighted directed forests, which are directed graphs whose components are directed trees. In Section \ref{sec:quasimetric-BA}, we present a representation theorem for quasimetric Boolean algebras. Section \ref{sec:Completeness-LAE} then establishes our final result, according to which provability in \pLAE{} coincides with the entailment relation of \LAE. Some concluding remarks are contained in Section \ref{sec:Conclusion}.

\section{The Logic of Approximate Entailment \\ in Quasimetric Spaces}
\label{sec:LAE}

We are concerned in this paper with a logic of approximate entailment; our work may be seen as a continuation of the contributions \cite{Rod,GoRo,EGRV} to this topic. The logics that have been considered under this name often differ in details. Here, our intention is to define and axiomatise a logic in close accordance with the original concepts of Ruspini, presented in his well-known paper \cite{Rus}.

Let us specify the semantic framework. The symbol $\ExtReals$ will denote the extended positive real line, that is, $\ExtReals = \Reals^+ \cup \{\infty\} = \{ r \in \Reals \colon r \geq 0 \} \cup \{\infty\}$. Here, $\infty$ is a new element such that $r < \infty$ for any $r \in \Reals^+$. We extend the addition from $\Reals^+$ to $\ExtReals$ requiring $r + \infty = \infty + r = \infty$ for any $r \in \ExtReals$. The operations of minimum and maximum on $\ExtReals$ will be denoted by $\wedge$ and $\vee$, respectively.

\begin{definition}
Let $W$ be a non-empty set. A {\it quasimetric} on $W$ is a mapping $q \colon W \times W \to \ExtReals$ such that for any $v, w, x \in W$:
\AxiomeAnfang
\Axiom{M1} $q(v,w) = 0$ if and only if $v = w$;
\Axiom{M2} $q(v,x) \leq q(v,w) + q(w,x)$.
\AxiomeEnde
In this case, we call $(W,q)$ a {\it quasimetric space}.

In a quasimetric space $(W,q)$, the {\it distance} of some $a \in W$ from some $B \subseteq W$ is defined to be
\[ q(a,B) \;=\; \inf_{b \in B} \; q(a,b). \]
Moreover, for $d \in \ExtReals$, we define the set
\[ U_d(B) \;=\; \{ a \in W \colon q(a,B) \leq d \} \]
as the {\it $d$-neighbourhood} of $B$. Finally, the {\it Hausdorff quasidistance} of some $A \subseteq W$ from some $B \subseteq W$ is
\[ q(A,B) \;=\; \sup_{a \in A} \; q(a,B). \]
\end{definition}

Note that we use the same symbol for the quasimetric, the distance of a point from a set, and the Hausdorff quasidistance.

A quasimetric space $(W,q)$ is called a {\it min-space} if, for any $a \in W$ and $B \subseteq W$, there is a $b \in B$ such that $q(a,B) = q(a,b)$. This notion is due to \cite{SSWZ}, where the importance of min-spaces in a logical setting has been pointed out.

Let us compile a few immediate observations.

\begin{lemma} \label{lem:quasimetric-space-1}
Let $(W,q)$ be a quasimetric space, let $a \in W$, $\;A, B \subseteq W$, and $c, d \in \ExtReals$.
\AufzAnfang

\Nummer{i} $q(a,B) = 0$ if $a \in B$. If $(W,q)$ is a min-space, also the converse holds.

\Nummer{ii} $q(a,B) \leq d$ if and only if $a \in U_d(B)$.

\Nummer{iii} $U_d(\emptyset) = \emptyset$ if $d < \infty$, and $U_\infty(\emptyset) = W$.

\Nummer{iv} $a \in U_d(B)$ if there is a $b \in B$ such that $q(a,b) \leq d$. If $(W,q)$ is a min-space, also the converse holds.

\Nummer{v} $U_c(U_d(A)) \subseteq U_{c+d}(A)$.

\AufzEnde
\end{lemma}

Moreover, the following properties are characteristic for the Hausdorff quasidistance.

\begin{lemma} \label{lem:quasimetric-space-2}
Let $(W,q)$ be a quasimetric space and let $A, B, C \subseteq W$.
\AufzAnfang

\Nummer{i} $q(A,B) = 0$ if $A \subseteq B$. If $(W,q)$ is a min-space, also the converse holds.

\Nummer{ii} $q(A,B) \leq d$ if and only if $A \subseteq U_d(B)$.

\Nummer{iii} $q(A \cup B, C) = q(A, C) \vee q(B, C)$.

\Nummer{iv} $q(A, B) \geq q(A, C)$ if $B \subseteq C$.

\Nummer{v} $q(A, C) \leq q(A, B) + q(B, C)$.

\AufzEnde
\end{lemma}

\begin{proof}
We only show (v). We calculate
\begin{align*}
q(A, B) + q(B, C)
& =\; \sup_{a \in A} \inf_{b \in B} q(a,b) + \sup_{b \in B} \inf_{c \in C} q(b,c) \\
& =\; \sup_{a \in A} \inf_{b \in B} (q(a,b) + \sup_{b' \in B} \inf_{c \in C} q(b',c)) \\
& \geq \sup_{a \in A} \inf_{b \in B} (q(a,b) + \inf_{c \in C} q(b,c)) \\
& =\; \sup_{a \in A} \inf_{b \in B} \inf_{c \in C} (q(a,b) + q(b,c)) \\
& \geq \sup_{a \in A} \inf_{c \in C} q(a,c) \;=\; q(A, C).&
\end{align*}
\end{proof}

We next specify the logic \LAE. The language of \LAE{} includes a countably infinite number of {\it variables}, denoted by small Greek letters, as well as the two constants $\false$ (false), and $\true$ (true). By a {\it Boolean formula}, we mean a formula built up from the variables and the constants by means of the binary operations $\land$ (and), and $\lor$ (or), as well the unary operation $\lnot$ (not). The connective $\impl$ (implies), is defined as usual.

Moreover, a {\it graded implication}, or an {\it implication} for short, is a triple consisting of two Boolean formulas $\alpha$ and $\beta$ as well as a number $d \in \Reals^+$. We write $\alpha \implc{d} \beta$ and we call $d$ the {\it degree} of the implication.

A {\it model} for \LAE{} is a Boolean algebra $({\mathcal B}; \cap, \cup, \complement, \emptyset, W)$ of subsets of a quasimetric space $(W, q)$ such that for each $A \in {\mathcal B}$ and $d \in \Reals^+$ also $U_d(A) \in {\mathcal B}$.

Given a model $\mathcal B$, an {\it evaluation} is a mapping $v$ from the Boolean formulas to $\mathcal B$ in the sense of classical propositional logic. That is, we require $v(\alpha \land \beta) = v(\alpha) \cap v(\beta)$, $\;v(\alpha \lor \beta) = v(\alpha) \cup v(\beta)$, $\;v(\lnot\alpha) = \complement v(\alpha)$, and $v(\false) = \emptyset$, $\;v(\true) = W$. An implication $\alpha \implc{d} \beta$ is said to be {\it satisfied} by an evaluation $v$ if
\[ v(\alpha) \subseteq U_d(v(\beta)). \]
A theory is a set of implications. We say that a theory $\mathcal T$ {\it semantically entails} some implication $\Phi$ in \LAE{} if, whenever all elements of $\mathcal T$ are satisfied by an evaluation $v$, also $\Phi$ is satisfied by $v$.

Having specified the logic \LAE{} in formal respects, let us now explain its intended meaning. First of all, we deal with Boolean formulas, interpreted by subsets of a fixed set of worlds. The Boolean formulas hence refer to crisp properties.

Furthermore, the role of the quasimetric is analogous to the role of the similarity relation in Ruspini's framework. A major difference is, however, that we do not assume symmetry here. Substituting the similarity relation by a quasimetric moreover means that the order is understood in the opposite way. Let $q$ be a quasimetric on a set of worlds $W$. Then, for two worlds $v$ and $w$, $\,q(v,w) = 0$ means coincidence, that is, $v = w$. Moreover, the larger $q(v,w)$ is, the less the worlds $v$ and $w$ resemble each other. Finally, $q(v,w) = \infty$ means that $v$ and $w$ are considered as unrelated.

The use of the dual order has practical reasons; the proofs will be easier to follow. A further advantage is that we are in line with the common topological terminology. We note that, however, a quasimetric can be identified with a quasisimilarity based on the product t-norm. Here, we understand a quasisimilarity similarly to a similarity relation \cite{Rus}, but without assuming symmetry. It is straightforward to see that under the correspondence $\;[0,1] \to \ExtReals \komma d \mapsto$ {\footnotesize $\begin{cases} - \ln d & \text{if $d \neq 0$,} \\ \infty & \text{if $d = 0$} \end{cases} \;$} the two concepts coincide.

Let us finally recall the meaning of an expression of the form $\alpha \implc{d} \beta$. This implication to be satisfied means that any world at which $\alpha$ holds is in the $d$-neighbourhood of the set of worlds at with $\beta$ holds. Conversely, if the implication is not satisfied, then there is a world $w$ with the following property: $\alpha$ holds at $w$, and for all worlds $v$ at which $\beta$ holds, we have $q(w,v) > d$. We conclude that the degree $d$ of the implication is a {\it degree of imprecision}: we may read $\alpha \implc{d} \beta$ as ``$\alpha$ implies $\beta$ within the limit of tolerance $d$''.

\section{The calculus \pLAE}
\label{sec:Rules}

The purpose of this paper is to define a proof system for \LAE. To this end we introduce a calculus denoted by \pLAE.

Recall that the syntactical objects of \LAE{} are graded implications and that we do not allow to build compound formulas from statements of the form (\ref{fml:LAE-formula}) by means of the classical connectives. Exactly this possibility is admitted in \cite{EGRV}, where a two-level language is used. Accordingly, the calculus of \cite{EGRV} is based on axioms together with the modus ponens as the only rule.

Our calculus \pLAE{} will instead be based on rules, suitable to derive implications from implications; there are no connectives on the metalevel. The style of \pLAE{} might be considered as somewhat similar to Skolem's calculus for lattice theory; see, e.g., \cite{NePl}. We are, in addition, inspired by the proof-theoretic methods in fuzzy logics \cite{MOG}; a remote resemblance with the sequent systems used in substructural logic might thus be observable as well.

We assume from now on that the set of variables is fixed and that from these variables, all Boolean formulas and implications are built up.

A {\it rule} of \pLAE{} consists of a possibly empty set of implications, called the {\it assumptions}, and one further implication, called the {\it conclusion}. We denote a rule writing the assumptions, if there are any, above the conclusion, using a separating horizontal line.

\begin{definition} \label{def:pLAE}
\pLAE{} consists of the following rules, where $\alpha$, $\beta$, $\gamma$ are any Boolean formulas and $c, d \in \Reals^+$:
\[ \Rule{R1} \alpha \implc{0} \beta \Add{ if $\alpha \impl \beta$ is a tautology of \CPL}
\quad\quad
\Rule{R2} \frac{\alpha \implc{0} \beta}{\alpha \land \gamma \implc{0} \beta \land \gamma} \]
\[ \Rule{R3} \frac{\alpha \implc{c} \beta}{\alpha \implc{d} \beta} \Add{where $d \geq c$}
\quad\quad
\Rule{R4} \frac{\alpha \implc{c} \false}{\alpha \implc{0} \false} \]
\[ \Rule{R5} \frac{\alpha \implc{c} \gamma \quad \beta \implc{c} \gamma}%
{\alpha \lor \beta \implc{c} \gamma}
\quad\quad
\Rule{R6} \frac{\alpha \implc{c} \beta \quad \beta \implc{d} \gamma}%
{\alpha \implc{c + d} \gamma} \]
Let $\mathcal T$ be a theory and $\Phi$ be an implication. A {\it proof} of $\Phi$ from $\mathcal T$ in \pLAE{} is defined in the usual way. We write ${\mathcal T} \proves \Phi$ if such a proof exists.

%
\end{definition}

Comparing \pLAE{} with the axiom system of \cite[Def.\ 4.2]{EGRV}, we may say that, roughly, \pLAE{} incorporates the same axioms except for those involving m.e.c.s, which are not available in the present context.

Our aim is to show the completeness of \pLAE{} with regard to the consequence relation of \LAE, that is, semantic entailment in \LAE{} coincides with provability in \pLAE. The soundness part does not cause difficulties.

\begin{proposition} \label{prop:soundness}
Let $\mathcal T$ be a theory and $\Phi$ an implication. If $\mathcal T$ proves $\Phi$ in \pLAE, then $\mathcal T$ semantically entails $\Phi$ in \LAE.
\end{proposition}

\begin{proof}
If the assumptions of a rule are satisfied by an evaluation $v$, the conclusion is satisfied by $v$ as well. This is clear for (R1)--(R5). In case of (R6), it follows from Lemma \ref{lem:quasimetric-space-1}(v).
\end{proof}

In spite of the seeming simplicity of the calculus \pLAE, the proof of the converse of Proposition \ref{prop:soundness}, given in Section \ref{sec:Completeness-LAE}, is not trivial. In the remainder of the present section, we will define a certain way of representing proofs of \pLAE.

We will call two Boolean formulas $\alpha$ and $\beta$ {\it Boolean equivalent}, in signs $\alpha \boolequ \beta$, if $\alpha \impl \beta$ and $\beta \impl \alpha$ are tautologies of \CPL. Assume $\alpha \boolequ \alpha'$ and $\beta \boolequ \beta'$; then the implications $\alpha \implc{d} \beta$ and $\alpha' \implc{d} \beta'$ are in \pLAE{} derivable from each other. In fact, in this case $\alpha' \impl \alpha$ and $\beta \impl \beta'$ are tautologies of \CPL, hence $\alpha' \implc{0} \alpha$ and $\beta \implc{0} \beta'$ follow by (R1). From $\alpha \implc{d} \beta$, we may then derive $\alpha' \implc{d} \beta'$ applying twice (R6). Similarly, we may infer $\alpha' \implc{d} \beta'$ from $\alpha \implc{d} \beta$. In the sequel, we will consider Boolean formulas up to Boolean equivalence. Boolean formulas will in fact be dealt with in disjunctive normal form.

\begin{definition}
A Boolean formula of the form $\phi$ or $\lnot\phi$, where $\phi$ is a variable, will be called a {\it literal}. By a {\it clause}, we mean a finite set of literals, and by a {\it clause set}, we mean a finite set of clauses. A clause $L$ is called {\it inconsistent} if there is a variable $\phi$ such that $L$ contains both $\phi$ and $\lnot\phi$, and {\it consistent} otherwise.

With a clause set
\begin{equation} \label{fml:block}
B \;=\; \{ \; \{ \lambda_{i1}, \ldots, \lambda_{ik_i} \} \colon 1 \leq i \leq l \},
\end{equation}
where $l \geq 0$ and $k_1, \ldots, k_l \geq 0$, we associate the Boolean formula
\[ \bigvee_{1 \leq i \leq l} \bigwedge_{1 \leq j \leq k_i} \lambda_{ij} \]
and denote it by $f(B)$. Here, an empty conjunction stands for $\true$ and an empty disjunction stands for $\false$.
\end{definition}

We note that, in this paper, clause sets represent disjunctive normal forms. That is, we interpret clauses conjunctively and sets of clauses disjunctively.

Clearly, each Boolean formula $\alpha$ is Boolean equivalent to a formula associated with a clause set. In particular, $\false$ is, e.g., associated with the empty set $\{ \}$; and $\true$ is, e.g., associated with the clause set consisting of a single empty clause, $\{ \{\} \}$. We will call a clause set $B$ such that $\alpha \boolequ f(B)$ a {\it clause set for} $\alpha$.

An implication of the following form is called {\it basic}:
\begin{equation} \label{fml:basic-implication}
\lambda_1 \land \ldots \land \lambda_n \implc{d}
\bigvee_{1 \leq i \leq l} \bigwedge_{1 \leq j \leq k_i} \mu_{ij};
\end{equation}
here, $n \geq 1$, $\;l \geq 0$, and $k_1, \ldots, k_l \geq 1$; moreover, $\{ \lambda_1, \ldots, \lambda_n \}$ and $\{ \mu_{i1}, \ldots, \mu_{ik_i} \}$, $i = 1, \ldots, l$, are consistent clauses. In particular, we may identify a basic implication with a triple consisting of a clause, a clause set, and a positive real.

A theory $\mathcal T$ consisting of basic implications will also be called {\it basic}. By the following lemma, we may restrict our considerations to basic theories.

\begin{lemma}
Let $\mathcal T$ be a theory. Then there is a basic theory ${\mathcal T}_b$ such that any implication is provable from $\mathcal T$ if and only if it is provable from ${\mathcal T}_b$.
\end{lemma}

\begin{proof}
Let us modify $\mathcal T$ as follows. Let $\alpha \implc{d} \beta$ be contained in $\mathcal T$. If $\lnot\alpha$ or $\beta$ is a tautology of \CPL, then $\alpha \implc{d} \beta$ is derivable in \LAE{} by (R1) and (R3). In this case, we drop it from $\mathcal T$. W.l.o.g., we can otherwise assume that the implication is of the form
\[ \bigvee_{1 \leq i \leq l'} \bigwedge_{1 \leq j \leq k'_i} \! \nu_{ij}
\;\;\implc{d}\;\;
\bigvee_{1 \leq i \leq l} \bigwedge_{1 \leq j \leq k_i} \mu_{ij}, \]
where $l' \geq 1$ and $l \geq 0$ and, for all $i$, $\{ \nu_{i1}, \ldots, \nu_{ik'_i} \}$, $\{ \mu_{i1}, \ldots, \nu_{ik_i} \}$ are non-empty consistent clauses. We drop in this case $\alpha \implc{d} \beta$ from $\mathcal T$ and add instead the basic implications
\begin{equation} \label{fml:basic-theory}
\bigwedge_{1 \leq j \leq k'_i} \!\!\! \nu_{ij} \;\implc{d}\; \bigvee_{1 \leq i \leq l} \bigwedge_{1 \leq j \leq k_i} \mu_{ij},
\end{equation}
where $i = 1, \ldots, l'$. By repeated application of (R5), we can derive $\alpha \implc{d} \beta$ from the implications (\ref{fml:basic-theory}). Conversely, by (R1) and (R6), we can derive the implication (\ref{fml:basic-theory}) for each $i$ from $\alpha \implc{d} \beta$.

Proceeding in the same way for each element of $\mathcal T$, we obtain a theory ${\mathcal T}_b$ that consists of basic implications only.
\end{proof}

In the sequel, we will identify proofs of \pLAE{} with certain weighted directed graphs labelled by Boolean formulas. We will adopt the common terminology to denote the relationship between nodes, like {\it child}, {\it father}, {\it sibling}, and {\it descendant}. Moreover, a {\it directed forest} is a directed graph such that each component is a directed tree; the latter's designated elements are called {\it roots} then. The nodes without child are called {\it terminal} nodes or {\it leaves}. A {\it subtree} as well as the subtree {\it rooted} at some node is understood in the obvious way. A {\it branch} is a maximal subtree with the property that each non-terminal node has exactly one child.

A {\it weighted} directed forest is a directed forest together with a map assigning to each edge a positive real number, called its {\it weight}. Finally, all graphs with which we deal will be labelled. However, to facilitate formulations, the mapping associating labels to nodes will be kept implicit; we will simply identify the nodes with their labels.

\begin{definition} \label{def:proof-forest}
A {\it proof forest} is a finite weighted directed forest each of whose nodes is either a clause or the symbol $\widerspruch$.

A node of a proof forest distinct from $\widerspruch$ is called {\it proper}. The {\it length} of a branch is defined to be $0$ if it contains $\widerspruch$ and else as the sum of the weights of its edges. By the length of a proof forest, we mean the maximum of the lengths of its branches.

Let $\mathcal T$ be a basic theory and let $\zeta \implc{r} \eta$ be an implication. Then a proof forest is called a {\it forest proof of $\zeta \implc{r} \eta$ from $\mathcal T$} if the following conditions hold:
\AxiomeAnfang
\Axiom{T1} There is a clause set $B_\zeta$ for $\zeta$ such that, for each clause $L$ in $B_\zeta$, there is a root that is a subset of $L$.

\Axiom{T2} There is a clause set $B_\eta$ for $\eta$ such that, for any terminal clause $L$, there is a clause in $B_\eta$ that is a subset of $L$.

\Axiom{T3} The length of the proof forest is at most $r$.

\Axiom{T4} Let $L$ be a non-terminal clause. Then the weights of all edges rooted at $L$ coincide; let $c$ be this value. One of the following possibilities applies:
\AxiomeAnfang
\Axiom{A} $c = 0$ and $\mathcal T$ contains a basic implication of the form (\ref{fml:basic-implication}), where $d = 0$, $\; \{ \lambda_1, \ldots, \lambda_n \} \subseteq L$, and for each $i = 1, \ldots, l$, a clause $L' \subseteq \{ \mu_{i1}, \ldots, \mu_{ik_i} \} \cup L$ is a child of $L$.

\Axiom{B} $c > 0$ and $\mathcal T$ contains a basic implication of the form (\ref{fml:basic-implication}), where $d = c$, $\; \{ \lambda_1, \ldots, \lambda_n \} \subseteq L$, and for each $i = 1, \ldots, l$, a clause $L' \subseteq \{ \mu_{i1}, \ldots, \mu_{ik_i} \}$ is a child of $L$.
\Axiom{C} $c = 0$, $L$ is inconsistent, and $\widerspruch$ is the only child of $L$.

\Axiom{D} $c = 0$, and for some variable $\phi$, called the {\it splitting} variable in the sequel, $L$ has exactly two children, one of which is a clause consisting of $\phi$ and a subset of $L$ and one of which is a clause consisting of $\lnot\phi$ and a subset of $L$.
\AxiomeEnde
\AxiomeEnde
\end{definition}

For illustration, let us include an example. Figure \ref{fig:Beispielbeweis} shows a forest proof of the implication
\[ \alpha \land \beta \;\implc{0.3}\; (\lnot\delta \land \epsilon) \lor (\delta \land \lnot\epsilon) \]
from the theory
\[ {\mathcal T} \;=\; \{ \alpha \implc{0} \lnot\beta \lor \gamma, \;\; \beta \land \gamma \implc{0.3} \delta \lor \epsilon, \;\; \delta \implc{0} \lnot\epsilon \}. \]
The block used for $\alpha \land \beta$ is $\{ \{ \alpha, \beta \} \}$; the block for $(\lnot\delta \land \epsilon) \lor (\delta \land \lnot\epsilon)$ is $\{ \{\lnot\delta, \epsilon \},$ $\{\delta, \lnot\epsilon\} \}$. Note that the length of the two branches not terminating with $\widerspruch$ is $0.3$; hence the length of the forest proof is $0.3$. We observe furthermore that all four cases of condition (T4) occur. For instance, at the root $\{ \alpha, \beta \}$, case (A) of (T4) applies; at its child $\{\beta, \gamma\}$, case (B) applies; at the root's child $\{ \beta, \lnot\beta \}$, case (C) applies; and at the node $\{ \delta \}$, case (D) applies.

\begin{figure}[ht!]
\begin{center}
\includegraphics[width=0.6\textwidth]{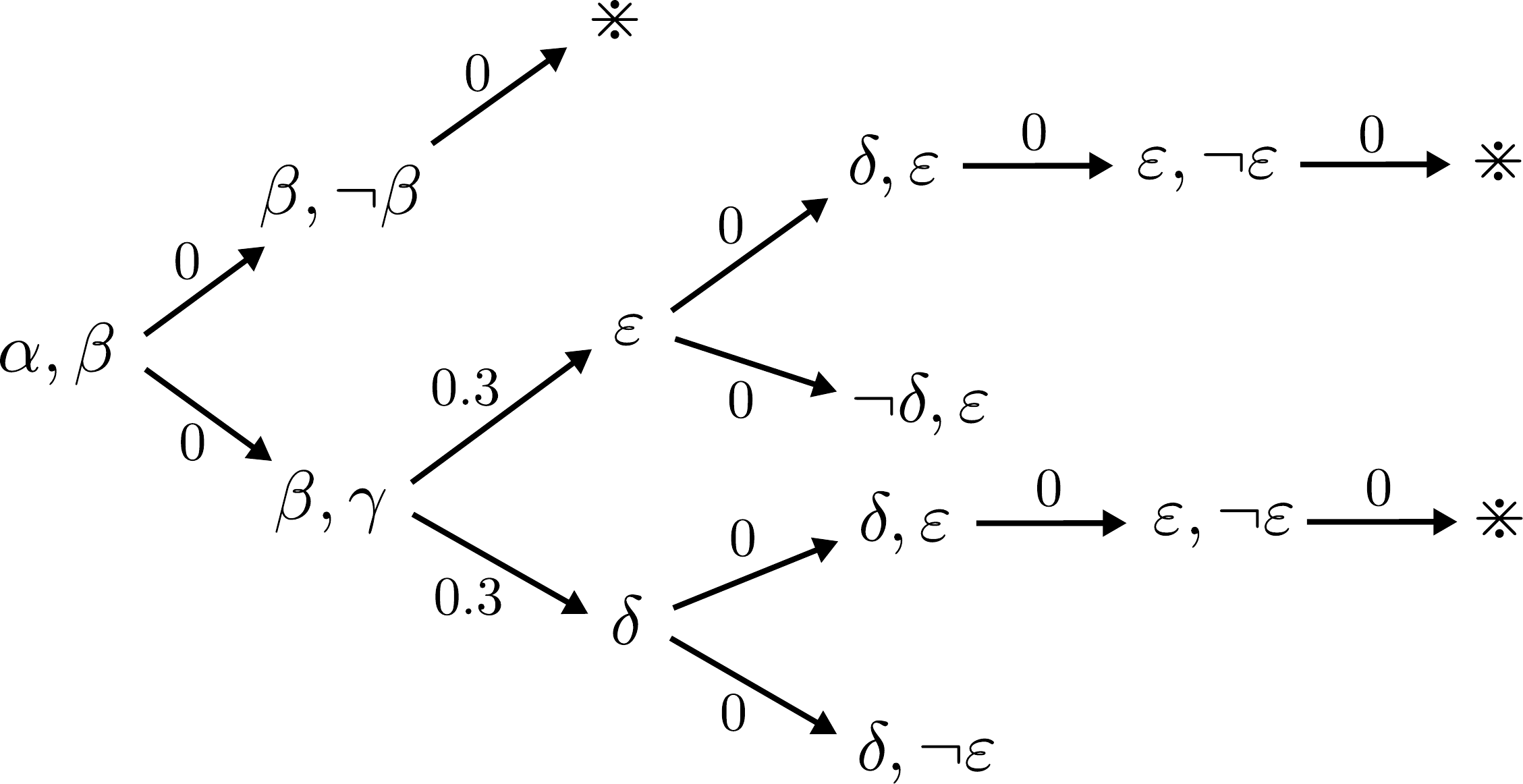}
\caption{An example of a forest proof. Set brackets are omitted.}
\label{fig:Beispielbeweis}
\end{center}
\end{figure}

The rest of this section is devoted to the aim of showing that provability in \pLAE{} is equivalent to the existence of proof forests.

We next show the first out of several technical lemmas, according to which a forest proof may be assumed to fulfil certain additional conditions.

Improper nodes are labelled by $\widerspruch$ and this symbol represents falsity. By the following lemma we can assume that the improper nodes are all terminal.

\begin{lemma} \label{lem:widerspruch-at-leaves-only}
Let there be a forest proof of $\zeta \implc{r} \eta$ from a basic theory $\mathcal T$. Then there is a forest proof of $\zeta \implc{r} \eta$ from $\mathcal T$ such that {\rm (i)} any node $\widerspruch$ is terminal and {\rm (ii)} a node $\widerspruch$ does not have siblings.
\end{lemma}

\begin{proof}
Consider a non-terminal node $\widerspruch$. We remove all children of $\widerspruch$ together with the subtrees rooted at them. Then the length of the proof forest is smaller than before. Moreover, each terminal node of the modified proof is a terminal node of the original proof; hence condition (T2) is still fulfilled. We conclude that the result is a forest proof of $\zeta \implc{r} \eta$ from $\mathcal T$. Applying the same procedure in all applicable cases, the first part follows.

Next, consider a clause $L$ that has a child $\widerspruch$ as well as further children. Then case (A) or (B) applies. Obviously, we can remove the subtree rooted at the child $\widerspruch$. Arguing as before, the second part follows as well.
\end{proof}

\begin{theorem} \label{thm:tob-proof-1}
Let $\mathcal T$ be a basic theory, and let $\zeta \implc{r} \eta$ be an implication. If there is a forest proof of $\zeta \implc{r} \eta$ from $\mathcal T$, then $\mathcal T \proves \zeta \implc{r} \eta$.
\end{theorem}

\begin{proof}
By Lemma \ref{lem:widerspruch-at-leaves-only}, we can assume that every node $\widerspruch$ is terminal and does not possess siblings.

Let $B_\zeta$ be a clause set for $\zeta$ according to (T1). If $B_\zeta$ is empty, we have $\zeta \boolequ \false$ and the assertion is clear from (R1) and (R3). Note that in this case the proof forest is arbitrary. For the rest of the proof, we assume that $B_\zeta$ is non-empty.

Let $\bar L$ be any clause contained in $B_\zeta$ and let us consider the tree whose root is a subset of $\bar L$. Our aim is to show $\mathcal T \proves \bigwedge \bar L \implc{r} \eta$. As $\zeta$ is Boolean equivalent to the disjunction of all conjunctions of a clause in $B_\zeta$, the assertion will then follow by (R1) and (R5). We claim:

($\star$) Let $L$ be a proper node. Let $S$ be the subtree rooted at $L$ and let $e$ be the length of $S$. Then $\bigwedge L \implc{e} \eta$ is provable in \pLAE{} from $\mathcal T$. Moreover, if all leaves of $S$ are $\widerspruch$, then $\bigwedge L \implc{0} \false$ is provable in \pLAE{} from $\mathcal T$.

Our assertion follows from ($\star$) applied to the root. Based on the following facts (a)--(c), ($\star$) follows by induction over the subtrees.

(a) Let the clause $L$ be a leaf. Then, by (T2), $\bigwedge L \implc{0} \eta$ is a tautology of \CPL{} and hence derivable by (R1).

(b) Let $L$ be a non-terminal clause and let $\widerspruch$ be its only child. We claim that then ${\mathcal T} \proves \bigwedge L \implc{0} \false$. Indeed, in case that (A) or (B) applies to $L$, $\mathcal T$ contains $\bigwedge L' \implc{c} \false$ for some $L' \subseteq L$ and $c \in \Reals^+$; by (R1) and (R4), the implication in question is derivable. In case (C), $\bigwedge L \boolequ \false$, and we draw the same conclusion. Case (D) cannot apply.

(c) Let $L$ be a non-terminal clause with the proper children $L_1, \ldots, L_n$. Let $c$ be the weight of the edges rooted at $L$. Assume that, for each $i = 1, \ldots, n$, ($\star$) holds for $L_i$. We claim that then ${\mathcal T} \proves \bigwedge L \implc{c} \bigvee_i \bigwedge L_i$. Indeed, in case that (A) applies to $L$, let $\alpha \implc{0} \beta$ be the used basic implication from $\mathcal T$. Then there is a conjunction $\gamma$ of literals such that $\alpha \land \gamma$ is Boolean equivalent with $\bigwedge L$ and $\beta \land \alpha \land \gamma \impl \bigvee_i \bigwedge L_i$ is a \CPL{} tautology. We apply (R2) to prove $\alpha \land \gamma \implc{0} \beta \land \alpha \land \gamma$ and by (R1), (R6) we derive $\bigwedge L \implc{0} \bigvee_i \bigwedge L_i$. In case (B), we proceed similarly; the application of (R2) being omitted. Case (C) cannot occur. In case (D), we use (R1).

Assume next that not all leaves of the subtree rooted at $L$ are $\widerspruch$. By assumption, we have, for all $i$, $\bigwedge L_i \implc{e_i} \eta$, where $e_i$ is the length of the subtree rooted at $L_i$. By (R5), we can derive $\bigvee_i \bigwedge L_i \implc{e} \eta$ as well, where $e$ is the largest value among the $e_i$, and by (R6) furthermore $\bigwedge L \implc{c+e} \eta$, and $c + e$ is the length of the subtree rooted at $L$.

Assume now that all leaves of the subtree rooted at $L$ are $\widerspruch$. By assumption, we then have, for all $i$, $\bigwedge L_i \implc{0} \false$. By (R5), (R6), and (R4), we derive $\bigwedge L \implc{0} \false$. In particular, we derive $\bigwedge L \implc{0} \eta$, and $0$ is the length of the subtree rooted at $L$. The proof of ($\star$) is complete.
\end{proof}

\begin{lemma} \label{lem:contradictory-clauses-end}
Let there be a forest proof of $\zeta \implc{r} \eta$ from a basic theory $\mathcal T$. Then there is a forest proof of $\zeta \implc{r} \eta$ from $\mathcal T$ such that each inconsistent clause is non-terminal and has the only child $\widerspruch$.
\end{lemma}

\begin{proof}
Let $\phi$ be a variable and let $L$ be a clause containing $\phi$ and $\lnot\phi$. If $L$ is non-terminal, we remove each child of $L$ together with the subtree rooted at it. Afterwards, we add $\widerspruch$ as the only child of $L$ and we endow the edge from $L$ to $\widerspruch$ with the weight $0$.

Then case (C) of condition (T4) applies at $L$. Furthermore, we readily check that the modified forest proof is still a forest proof of $\zeta \implc{r} \eta$ from $\mathcal T$.
\end{proof}

\begin{theorem} \label{thm:tob-proof-2}
Let $\mathcal T$ be a basic theory and let $\zeta \implc{r} \eta$ be an implication. If $\mathcal T \proves \zeta \implc{r} \eta$, then there is a forest proof of $\zeta \implc{r} \eta$ from $\mathcal T$.
\end{theorem}

\begin{proof}
By (T4), cases (A) or (B), each implication contained in $\mathcal T$ possesses a forest proof.

We will now check for each of the rules (R1)--(R6) that whenever each assumption possesses a forest proof, so does the conclusion. The assertion will then follow by induction over the length of a proof in \pLAE.

Ad (R1). Let $\alpha \impl \beta$ be a tautology of \CPL. Then there is a clause set $B_\alpha = \{ L_1, \ldots, L_n \}$ for $\alpha$ and a clause set $B_\beta$ for $\beta$ such that, for each $i$, $L_i$ is a subset of a clause in $B_\beta$. Then the forest that consists of the single node $\widerspruch$ if $B_\alpha$ is empty, and of the isolated nodes $L_i$ otherwise is a forest proof of $\alpha \implc{0} \beta$.

Ad (R2). Assume that there is a forest proof of $\alpha \implc{0} \beta$. Let $\gamma$ be any Boolean formula; then there is also a forest proof of $\alpha \land \gamma \implc{0} \beta \land \gamma$, as follows from the following two observations.

First, let $\lambda$ be a literal; we construct a forest proof of $\alpha \land \lambda \implc{0} \beta \land \lambda$ as follows. If $\alpha \boolequ \false$, the forest proof consisting of the node $\widerspruch$ alone fulfils the requirement. Otherwise, we add $\lambda$ to the (proper) roots; and for all proper nodes $L$ based on case (A) or (D), we recursively add $\lambda$ to the proper children of $L$ if $\lambda$ was added to $L$. The result fulfils the requirements; for, if a node $L$ is based on case (B) or (C), all branches through $L$ end up with $\widerspruch$.

Second, assume that, for Boolean formulas $\gamma_1$ and $\gamma_2$, there are forest proofs of $\alpha \land \gamma_1 \implc{0} \beta \land \gamma_1$ as well as $\alpha \land \gamma_2 \implc{0} \beta \land \gamma_2$. Joining both these proofs, we get a forest proof of $\alpha \land (\gamma_1 \lor \gamma_2) \implc{0} \beta \land (\gamma_1 \lor \gamma_2)$.

Ad (R3). In view of (T3), a forest proof of $\alpha \implc{c} \beta$ is also a forest proof of $\alpha \implc{d} \beta$, where $d \geq c$.

Ad (R4). Let a forest proof of $\alpha \implc{c} \false$ be given. By Lemma \ref{lem:contradictory-clauses-end} we can assume that no leaf is inconsistent. We conclude that all leaves are $\widerspruch$. Hence the proof forest is actually a forest proof of $\alpha \implc{0} \false$.

Ad (R5). Assume there are forest proofs of $\alpha \implc{c} \gamma$ and $\beta \implc{c} \gamma$. Then their union is a forest proof of $\alpha \lor \beta \implc{c} \gamma$.

Ad (R6). Assume there are forest proofs of $\alpha \implc{c} \beta$ and $\beta \implc{d} \gamma$. By Lemma \ref{lem:contradictory-clauses-end}, we can assume that no terminal clause is inconsistent.

Assume first that $\beta \boolequ \false$. Then, as we have argued in case of the rule (R4), all leaves of the first proof forest are $\widerspruch$. Hence the first proof forest is a also forest proof of $\alpha \implc{c} \gamma$ and hence of $\alpha \implc{c+d} \gamma$.

Assume second that $\beta$ is not Boolean equivalent to $\false$. Let $L$ be a proper leaf of the first proof. Then we can enlarge the proof forest by applications of (T4)(D) to $L$ such that each new leaf $L'$ includes a root $\bar L$ of the second proof. We then connect all children of $\bar L$, together with the subtrees rooted at them, to $L'$. Proceeding in the same way for all proper leafs of the first proof, we obtain a forest proof of $\alpha \implc{c + d} \gamma$.
\end{proof}

We summarise that \pLAE{} proves an implication $\zeta \implc{r} \eta$ from a theory $\mathcal T$ exactly if there exists a forest proof of $\zeta \implc{r} \eta$ from $\mathcal T$.

\section{Quasimetric Boolean algebras}
\label{sec:quasimetric-BA}

The propositions of our logic are interpreted by subsets of a quasimetric space. Moreover, the degree to which an implication is defined to hold is determined by the Hausdorff quasidistance, which assigns to any two elements of this algebra a positive real.

In this context we are led to the following notion. Recall that a Boolean algebra $\mathcal B$ is called {\it separable} if $\mathcal B$ is generated by a countable subset.

\begin{definition}
A {\it quasimetric Boolean algebra} is a pair $({\mathcal B}, p)$ consisting of a separable Boolean algebra $({\mathcal B}; \wedge, \vee, \compl, 0, 1)$ and a mapping $p \colon {\mathcal B} \times {\mathcal B} \to \ExtReals$ such that, for all $A, B, C \in {\mathcal B}$,
\AxiomeAnfang
\Axiom{QB1} $p(A,B) = 0$ if and only if $A \leq B$.

\Axiom{QB2} $p(A \vee B, C) = p(A, C) \vee p(B, C)$.

\Axiom{QB3} $p(A, B) \leq p(A, C)$ if $C \leq B$.

\Axiom{QB4} $p(A, C) \leq p(A, B) + p(B, C)$.
\AxiomeEnde
\end{definition}

We note that the present notion should not be confused with the notion of a metric Boolean algebra, which is usually meant to be a Boolean algebra together with a measure on it.

The aim of this section is to establish a representation theorem for quasimetric Boolean algebras. The Boolean algebra $\mathcal B$ will be represented by a set of subsets of a set $W$ and the quasimetric $p$ on $\mathcal B$ by the Hausdorff quasidistance induced by a quasimetric on $W$.

We will need in the sequel Stone's representation theorem of Boolean algebras \cite{Sik}. Recall that a subset $F$ of a Boolean algebra $\mathcal B$ is called a prime filter if (i) $a \in F$ and $a \leq b$ imply $b \in F$, (ii) $a \wedge b \in F$ for any $a, b \in F$, and (iii) for each $a \in {\mathcal B}$, exactly one of $a$ and $\compl a$ is in $F$.

\begin{theorem}
Let $({\mathcal B}; \wedge, \vee, \compl, 0, 1)$ be a Boolean algebra. Then there is a compact topological space $W$ such that $\mathcal B$ is isomorphic with $({\mathcal C}; \cap, \cup, \complement, \emptyset, W)$, where $\mathcal C$ is the set of clopen subsets of $W$.
\end{theorem}

\begin{proof}[Proof (sketched)]
Let $W$ be the set of prime filters of $\mathcal B$. Define $\iota \colon {\mathcal B} \to {\mathcal P}(W) \komma a \mapsto \{ F \in W \colon a \in F \}$. Then $\iota$ is an injective homomorphism from $\mathcal B$ to ${\mathcal P}(W)$, endowed with the set-theoretical operations.

Let furthermore $W$ be endowed with the coarsest topology whose open sets include $\iota(a)$ for all $a \in {\mathcal B}$. Then the image of $\mathcal B$ is the set of clopen subsets of $W$. Hence $\iota$ induces an isomorphism as claimed.
\end{proof}

Our representation theorem will not apply for quasimetric Boolean algebras in general; we will assume two further properties.

\begin{theorem} \label{thm:metricBA-representation}
Let $({\mathcal B}, p)$ be a quasimetric Boolean algebra such that the following additional conditions are fulfilled.
\AxiomeAnfang
\Axiom{QB5} For all $A, B, C \in {\mathcal B}$, there are $A_B, A_C \in {\mathcal B}$ such that $A = A_B \vee A_C$ and $p(A, B \cup C) = p(A_B, B) \vee p(A_C, C)$.

\Axiom{QB6} For any sequence $(A_i)_i$, from $\inf_i p(A_i, B) = 0$ it follows $A_j \leq B$ for some $j$.
\AxiomeEnde
Then there exists a quasimetric space $(W, q)$ and an injective Boolean homomorphism $\iota$ from $\mathcal B$ to the algebra of subsets of $W$ such that
\[ p(A,B) \;=\; q(\iota(A), \iota(B)) \]
for all $A, B \in {\mathcal B}$.
\end{theorem}

\begin{proof}
Let $W$ be the compact topological space consisting of the prime filters of $\mathcal B$, such that $\mathcal B$ can be identified with the set of clopen sets endowed with the set-theoretical operations. Note that, by construction, each sequence $A_0 \supseteq A_1 \supseteq \ldots$ of elements of $\mathcal B$ has a non-empty intersection; and conversely, as $\mathcal B$ is separable, there is for each $w \in W$ a sequence $A_0 \supseteq A_1 \supseteq \ldots$ in $\mathcal B$ such that $\bigcap_i A_i = \{w\}$.

For $a \in W$ and $B \in {\mathcal B}$, we put
\begin{equation} \label{fml:q-1}
q(a,B) \;=\; \inf_{{a \in A,} \atop {A \in {\mathcal B}}} p(A,B),
\end{equation}
and for $a, b \in W$, we put
\begin{equation} \label{fml:q-2}
q(a,b) \;=\; \sup_{{b \in B,} \atop {B \in {\mathcal B}}} q(a,B).
\end{equation}
In the rest of the proof, we will omit the explicit reference to $\mathcal B$; capital Latin letters will always vary over elements of $\mathcal B$. Our first aim is to show that $q \colon W \times W \to \ExtReals$ is a quasimetric.

To see (M1), let $a, b \in W$ such that $q(a,b) = 0$. Let $B$ be such that $b \in B$. Then $\inf_{a \in A} p(A,B) = 0$. Hence there is a sequence $A_0 \supseteq A_1 \supseteq \ldots$ such that $a \in A_i$ for all $i$ and $\inf_i p(A_i, B) = 0$. We conclude from (QB6) that $p(A_j, B) = 0$ for some $j$; from (QB1) that $A_j \subseteq B$; and hence $a \in B$. As $B$ is an arbitrary set containing $b$, it follows $a = b$.

To see (M2), let $a, b, c \in W$. We calculate
\begin{align*}
q(a,b) + q(b,c) \;
& =\; \sup_{b \in B} \inf_{a \in A} p(A, B) + \sup_{c \in C} \inf_{b \in B'} p(B', C) \\
& =\; \sup_{c \in C} \inf_{b \in B'} (\sup_{b \in B} \inf_{a \in A} p(A, B) + p(B', C)) \\
& \geq \sup_{c \in C} \inf_{b \in B'} (\inf_{a \in A} p(A, B') + p(B', C)) \\
& =\; \sup_{c \in C} \inf_{b \in B'} \inf_{a \in A} (p(A, B') + p(B', C)) \\
& \geq \sup_{c \in C} \inf_{a \in A} p(A, C) \\
& =\; q(a,b).
\end{align*}
It remains to prove that
\begin{equation} \label{fml:qAB}
p(A,B) \;=\; \sup_{a \in A} \inf_{b \in B} q(a,b)
\end{equation}
holds for any $A, B \in {\mathcal B}$. To this end, we first prove some auxiliary facts.

(a) Let $a \in W$ and $B, C \in {\mathcal B}$. Then $q(a, B \cup C) = q(a,B) \wedge q(a,C)$.

By (QB3), $q(a, B \cup C) \leq q(a, B) \wedge q(a, C)$. Let $\epsilon > 0$ and choose some $A$ such that $a \in A$ and $q(a, B \cup C) \leq q(A, B \cup C) \leq q(a, B \cup C) + \epsilon$. By (QB5), there are $A', A''$ such that $A = A' \cup A''$ and $q(A', B) \vee q(A'', C) = q(A, B \cup C)$. If $a \in A'$, we have $q(a, B) \leq q(A', B) \leq q(A, B \cup C) \leq q(a, B \cup C) + \epsilon$. Taking into account the possibility $a \in A''$ as well, we conclude $q(a, B) \wedge q(a, C) \leq q(a, B \cup C) + \epsilon$. As $\epsilon$ is arbitrary, (a) follows.

(b) Let $a \in W$ and $B \in {\mathcal B}$. Then $q(a,B) = \inf_{b \in B} q(a,b)$.

Clearly, $q(a,B) \leq q(a,b)$ for any $b \in B$. Assume that $c \leq q(a,b)$ for any $b \in B$. Let $\epsilon > 0$. For each $b \in B$, we have $q(a,b) = \sup_{b \in C} q(a, C) = \sup_{b \in C, \; C \subseteq B} q(a, C)$ and hence we may choose a $C_b \subseteq B$ containing $b$ such that $c \leq q(a, C_b) + \epsilon$. By compactness, there are finitely many elements $b_1, \ldots, b_k$ such that $C_{b_1} \cup \ldots \cup C_{b_k} = B$. We conclude $c \leq q(a, C_{b_1}) \wedge \ldots \wedge q(a, C_{b_k}) + \epsilon = q(a, B) + \epsilon $ by (a). As $\epsilon$ is arbitrary, $c \leq q(a,B)$, and (b) follows.

(c) Let $A, B \in {\mathcal B}$. Then there is an $\bar a \in A$ such that $p(A, B) = q(\bar a, B)$.

By the separability of $\mathcal B$ and by (QB2), there is a sequence $A = A_0 \supset A_1 \supset \ldots$ in $\mathcal B$ such that $\bigcap_i A_i = \{\bar a\}$ for some $\bar a \in W$ and $p(A, B) = p(A_i, B)$ for any $i$. Then $q(\bar a, B) = \inf_{\bar a \in A'} p(A', B) = \inf_i p(A_i, B) = p(A, B)$ and (c) is shown.

We finally turn to the proof of (\ref{fml:qAB}). Let $A, B \in {\mathcal B}$. By (c), there is an $\bar a \in A$ such that $p(A, B) = q(\bar a, B)$; consequently, $p(A, B) = \max_{a \in A} q(a, B)$. By (b), we have $q(a,B) = \inf_{b \in B} q(a,b)$ for each $a \in A$. The assertion follows.
\end{proof}

\section{A completeness theorem for \LAE}
\label{sec:Completeness-LAE}

This section contains the main result of the paper: we show that the calculus \pLAE{} is complete with respect to the entailment relation of \LAE.

The proof of the completeness theorem will rely on the notion of a proof forest, defined in Section \ref{sec:Rules}, as well as on the representation of quasimetric Boolean algebras, contained in Section \ref{sec:quasimetric-BA}. In addition, we will need a couple of auxiliary results on proof forests (Lemma \ref{lem:smallest-weight}--\ref{lem:smallest-weight-whole}).

\begin{lemma} \label{lem:smallest-weight}
Let $\mathcal T$ be a finite theory and let $\zeta, \eta$ be Boolean formulas such that ${\mathcal T} \proves \zeta \implc{s} \eta$ for some $s \in \Reals^+$. Then there is an $r \in \Reals^+$ with the following properties: {\rm (i)} $r$ is the smallest value such that ${\mathcal T} \proves \zeta \implc{r} \eta$ and {\rm (ii)} $r$ is a finite sum of weights of implications in $\mathcal T$.
\end{lemma}

\begin{proof}
Let \pLAE' be the calculus obtained modifying \pLAE{} in the following way: rule (R3) is dropped; and (R5) is replaced by
\[ \Rule{R5'} \frac{\alpha \implc{c} \gamma \quad \beta \implc{d} \gamma}%
{\alpha \lor \beta \implc{c \vee d} \gamma}. \]
Assume that, for some $s \in \Reals^+$, there is a proof of $\zeta \implc{s} \eta$ from $\mathcal T$ in \pLAE. From this proof, we drop every application of (R3) and decrease the weights of the subsequent implications accordingly, such that a proof in \pLAE' is obtained. Let $\zeta \implc{s'} \eta$ be the proved implication. Then $s' \leq s$ and $s'$ is the sum of finitely many reals each of which is the weights of some implication contained in $\mathcal T$. This proof may in turn transformed to a proof of $\zeta \implc{s'}$ in \pLAE. Indeed, we replace each application of (R5') by an application of (R3) followed by (R5).

We conclude that if $\zeta \implc{s} \eta$ is provable, then so is $\zeta \implc{s'} \eta$, where $s'$ is a sum of elements of a finite set of positive reals. The assertions follow.
\end{proof}

In the context of a forest proof, we will use the following additional notions. For a non-terminal clause $L$, we say that a literal $\lambda$ is {\it introduced} at $L$ if $L$ does not contain $\lambda$, but some child of $L$ does. Conversely, if $L$ contains $\lambda$ but no child of $L$ does so, we say that $\lambda$ is {\it dropped} at $L$.

Furthermore, assume that the literal $\lambda$ is contained in the non-terminal clause $L$ and one of the following conditions holds:
\vspace{-1ex}
\begin{enumerate}

\item Case (A) or (B) of condition (T4) applies at $L$. Moreover, with reference to the notation of Definition \ref{def:proof-forest}, $\lambda$ is not among $\lambda_1, \ldots, \lambda_n$. Finally, if $\lambda$ is contained in the child of $L$ indexed by $i \in \{1, \ldots, l\}$, $\lambda$ is one of $\mu_{i1}, \ldots, \mu_{im_i}$.

\item Case (C) applies at $L$. Moreover, $L \backslash \{\lambda\}$ is still inconsistent.

\item Case (D) applies at $L$. Moreover, $\lambda$ is not contained in any children of $L$.

\end{enumerate}
\vspace{-1ex}
Then we call $\lambda$ {\it unused} in $L$.

\begin{lemma} \label{lem:no-unused-literal}
Let there be a forest proof of $\zeta \implc{r} \eta$ from $\mathcal T$. Then there is a forest proof of $\zeta \implc{r} \eta$ from $\mathcal T$ such that there is no unused literal in any proper node.
\end{lemma}

\begin{proof}
Let the non-terminal clause $L$ contain the unused literal $\lambda$. We remove $\lambda$ from $L$. Then, evidently, condition (A), (B), (C), or (D), respectively, applies to the new node $L \backslash \{\lambda\}$ if this was the case before.

If $L \backslash \{\lambda\}$ is the root, we are done. Otherwise, let $L'$ be the father of $L \backslash \{\lambda\}$. $L'$ still fulfils the requirements of (T4) except for the case that, in the original proof forest, case (D) is applied to $L'$ and $\lambda$ is one of $\phi$ or $\lnot\phi$, where $\phi$ is the splitting variable. Then $L'$ has the child $L \backslash \{ \lambda \} \subseteq L'$ as well as a second child that contains the negation of $\lambda$. We remove the latter together with the subtree rooted at it; and we remove the node $L \backslash \{ \lambda \}$ and connect its children directly to $L'$. Evidently, the result is again a forest proof of $\zeta \implc{r} \eta$ from $\mathcal T$.

We repeat the same procedure as long as it is applicable.
\end{proof}

Note what it means that a forest proof does not have any unused variables. Namely, let $L$ be a non-terminal clause in such a forest proof. In case (A), if $L$ contains any literal $\lambda$ in addition to $\lambda_1, \ldots, \lambda_n$, then $\lambda$ is also contained in a child, where $\lambda$ is not among the respective literals $\mu_{i1}, \ldots, \mu_{ik_i}$. In case (B), we have $L \;=\; \{ \lambda_1, \ldots, \lambda_n \}$. In case (C), $L = \{ \phi, \lnot\phi \}$ for some variable $\phi$. Finally, in case (D), each $\lambda \in L$ is also contained in one of the two children of $L$.

By a {\it standard clause set} for a Boolean formula $\alpha$, we mean a clause set $B$ such that, for any variable $\phi$, the following holds: if $\phi$ is contained in a clause of $B$, then $\phi$ occurs in $\alpha$ positively, and if $\lnot\phi$ is contained in a clause of $B$, then $\phi$ occurs in $\alpha$ negatively.

\begin{lemma} \label{lem:no-extra-literal-in-root-and-leaves}
Let there be a forest proof of $\zeta \implc{r} \eta$ from $\mathcal T$. Then there is a forest proof of $\zeta \implc{r} \eta$ from $\mathcal T$ such that the following holds. For each variable $\phi$ and each root $L$, $\phi$ is contained in $L$ only if $\phi$ positively occurs in $\zeta$, and $\lnot\phi$ is contained in $L$ only if $\phi$ negatively occurs in $\zeta$. Similarly, for each variable $\phi$ and each leaf $L$, $\phi$ is contained in $L$ only if $\phi$ positively occurs in $\eta$, and $\lnot\phi$ is contained in $L$ only if $\phi$ negatively occurs in $\eta$.
\end{lemma}

\begin{proof}
Let $B_\zeta$ be a standard clause set for $\zeta$. We construct a new forest proof as follows. Let $L$ be a clause in $B_\zeta$. Then $L$ is chosen as a root of the new proof. We apply to $L$ (T4)(D) as many times as necessary such that eventually each leaf includes a root of the original proof. For each such leaf $L'$, we connect to $L'$ all children of the corresponding root of the original proof, together with the subtrees rooted at them. Proceeding in the same way for each clause in $B_\zeta$, the requirement of the first part of the lemma are fulfilled.

Let $B_\eta$ be a standard clause set for $\eta$. Then we apply (T4)(D) at each leaf as many times as necessary such that each leaf $L$ includes a clause $\bar L$ in $B_\eta$. Next we replace $L$ with its subset $\bar L$. If, in the original forest, (T4)(D) applied to the father $L'$ of $L$ and the splitting variable is now missing, we remove all descendants of $L'$, and we apply the same procedure to $L'$, which then becomes a leaf including $\bar L$. In this way, we successively fulfil also the second part.
\end{proof}

\begin{lemma} \label{lem:case-distinction-only-for-new-variables}
Let there be a forest proof of $\zeta \implc{r} \eta$ from $\mathcal T$. Then there is a forest proof such that the following holds. Let $L$ be a non-terminal clause to which case {\rm (D)} of condition {\rm (T4)} applies. Then the splitting variable $\phi$ does not occur in $L$.
\end{lemma}

\begin{proof}
Assume that $\phi$ does occur in $L$, that is, either $\phi$ or $\lnot\phi$ is in $L$. Let then $L'$ be the child of $L$ that is included in $L$. We remove all descendants of $L$ and we connect instead the children of $L'$ directly to $L$.

Proceeding in the same way in all applicable cases, we fulfil the requirements of the lemma.
\end{proof}

\begin{lemma} \label{missing-literal-not-in-tob-proof}
Let there be a forest proof of $\zeta \implc{r} \eta$ from $\mathcal T$. Let $\phi$ be a variable that does not occur negatively in $\zeta$ or in $\eta$; and for any implication $\alpha \implc{d} \beta$ contained in $\mathcal T$, $\phi$ does not occur negatively in $\alpha$ or $\beta$. Then there is a forest proof in which the literal $\lnot\phi$ does not occur.

A similar statement holds with respect to positive occurrences of $\phi$.
\end{lemma}

\begin{proof}
We will show the first part; for the second part we can argue similarly. Accordingly, assume that $\phi$ does not occur negatively in $\zeta$ or $\eta$, and for any element of $\mathcal T$ given in the form (\ref{fml:basic-implication}), assume that the literal $\lnot\phi$ does not occur.

Applying successively Lemmas \ref{lem:contradictory-clauses-end}, \ref{lem:case-distinction-only-for-new-variables}, \ref{lem:no-extra-literal-in-root-and-leaves}, and \ref{lem:no-unused-literal}, we can assume the following additional properties of the forest proof under consideration: any inconsistent clause has the only child $\widerspruch$; (T4)(D) is not used when the splitting variable is already present; $\lnot\phi$ does not occur in any root or leaf; and no unused literal occurs.

Assume that $\lnot\phi$ occurs in the forest proof. Let $L$ be a clause at which $\lnot\phi$ is introduced such that $\lnot\phi$ is not introduced at any descendant of $L$. At $L$, only case (D) can apply; hence $L$ has two children $L_1$ and $L_2$ such that $\phi \in L_1$ and $\lnot\phi \in L_2$. Let $S$ be the subtree whose root is $L_2$ and that is maximal w.r.t.\ the property that its nodes are clauses containing $\lnot\phi$. At the non-terminal nodes of $S$, then only cases (A) or (D) can apply. Furthermore, at each leaf $L'$ of $S$, $\lnot\phi$ is by assumption dropped. This means that case (C) is applied to $L' = \{ \phi, \lnot\phi \}$.

We now modify the proof forest with the effect that $\lnot\phi$ will no longer occur in any descendant of $L$. We remove the subtree rooted at $L_1$; and from each clause in $S$, we drop $\lnot\phi$ and add the elements of $L$ instead. The former node $L_2$ becomes $L$; we replace the two coinciding nodes by a single one. Moreover, each leaf $L'$ of $S$ is now $L \cup \{ \phi \}$ and hence includes $L_1$; we remove its only child $\widerspruch$ and connect the subtrees rooted at the (former) children of $L_1$, if there are any.

Since each weight in the subtree $S$ is $0$, the length of the modified proof forest is at most the length of the original one; thus we obtain a proof forest of $\zeta \implc{r} \eta$ from $\mathcal T$ again. We proceed in the same way as long as the literal $\lnot\phi$ is present.
\end{proof}

It is the next lemma where we make use of the special properties of proof forests that we have shown to be assumable.

\begin{lemma} \label{lem:extension-single}
Let $\mathcal T$ be a basic theory. Then there is a basic theory ${\mathcal T}'$ including $\mathcal T$ with the following property:
\AxiomeAnfang

\Axiom{E1} Let $\alpha, \beta, \gamma$ be Boolean formulas built up from variables that occur in $\mathcal T$. If ${\mathcal T} \proves \alpha \implc{d} \beta \lor \gamma$ for some $d \in \Reals^+$, then there are Boolean formulas $\alpha_\beta, \alpha_\gamma$ such that the following is provable from ${\mathcal T}'$:
\begin{equation} \label{fml:disjunction}
\alpha_\beta \implc{d} \beta, \quad \alpha_\gamma \implc{d} \gamma, \quad
\alpha_\beta \implc{0} \alpha, \quad \alpha_\gamma \implc{0} \alpha, \quad
\alpha \implc{0} \alpha_\beta \lor \alpha_\gamma.
\end{equation}

\Axiom{E2} Let $\zeta$ and $\eta$ be Boolean formulas built up from variables that occur in $\mathcal T$. If ${\mathcal T}' \proves \zeta \implc{r} \eta$, then ${\mathcal T} \proves \zeta \implc{r} \eta$.
\Axiom{E3} Every weight of an implication in ${\mathcal T}'$ is a finite sum of weights of implications in $\mathcal T$.
\AxiomeEnde
\end{lemma}

\begin{proof}
Let $(\alpha, \beta, \gamma)$ be a triple of Boolean formulas built up from variables occurring in $\mathcal T$, such that $\alpha$ is the conjunction of literals. Assume that there is a $d \in \Reals^+$ such that ${\mathcal T} \proves \alpha \implc{d} \beta \lor \gamma$. Assume furthermore that $d$ is the smallest weight with this property; such a weight exists by Lemma \ref{lem:smallest-weight} and is a sum of weights of implications contained in $\mathcal T$. Assume finally that there are no two Boolean formulas $\alpha_\beta, \alpha_\gamma$ such that all five implications (\ref{fml:disjunction}) are provable from $\mathcal T$. Note that then $d > 0$.

Let then $\alpha_\beta$ and $\alpha_\gamma$ be variables that are not among those occurring in $\mathcal T$ and add the five basic implications (\ref{fml:disjunction}) to $\mathcal T$. Let ${\mathcal T}'$ arise from $\mathcal T$ by proceeding for each such triple in the indicated way, in each case using two new variables.

Then, by construction, ${\mathcal T}'$ fulfils (E1) restricted to conjunctions of literals $\alpha$ and minimal weights $d$. We readily check that ${\mathcal T}'$ actually fulfils (E1) in general. It is furthermore clear that ${\mathcal T}'$ fulfils also (E3).

To see that ${\mathcal T}'$ fulfils (E2), let $\zeta$ and $\eta$ contain only variables present in $\mathcal T$, and assume that ${\mathcal T}' \proves \zeta \implc{r} \eta$. By Theorem \ref{thm:tob-proof-2}, there is a forest proof of $\zeta \implc{r} \eta$ from ${\mathcal T}'$.

Let $(\alpha, \beta, \gamma)$ be one of the triples that gave rise to an extension of $\mathcal T$; let $\alpha$ be the conjunction of the literals $\lambda^\alpha_1, \ldots, \lambda^\alpha_n$, $n \geq 1$. That is, ${\mathcal T}'$ contains the implications (\ref{fml:disjunction}) and the variables $\alpha_\beta$ and $\alpha_\gamma$ occur exactly in these five implications. Assume that at least one of these variables occurs in the forest proof. We shall show that we can eliminate $\alpha_\beta$ and $\alpha_\gamma$. We can proceed in the same way for every pair of newly added variables and (E2) will follow.

By Lemma \ref{missing-literal-not-in-tob-proof}, we can assume that the proof forest does not contain the literals $\lnot\alpha_\beta$ and $\lnot\alpha_\gamma$. By Lemma \ref{lem:no-extra-literal-in-root-and-leaves}, we can furthermore assume that the variables $\alpha_\beta$ and $\alpha_\gamma$ do not occur in the roots and leaves. By Lemma \ref{lem:no-unused-literal}, we may finally assume that there is no unused literal.

Let $L$ be a clause at which $\alpha_\beta$ or $\alpha_\gamma$ is introduced. Then case (A) of condition (T4) applies to $L$ and $\alpha \implc{0} \alpha_\beta \vee \alpha_\gamma$ is the used implication. In particular, $L \supseteq \{\lambda^\alpha_1, \ldots, \lambda^\alpha_n\}$ and we can assume that $L$ has exactly two children: the child $L_1$ consisting of $\alpha_\beta$ and a subset of $L$, as well as the child $L_2$ consisting of $\alpha_\gamma$ and a subset of $L$. Let $S_1$ be the subtree whose root is $L_1$ and that is maximal w.r.t.\ the property that each node is proper and contains $\alpha_\beta$. Similarly, let $S_2$ be the subtree $S_2$ whose root is $L_2$ and that is maximal w.r.t.\ the property that each node is proper and contains $\alpha_\gamma$. Note that all weights in $S_1$ and $S_2$ are $0$. Moreover, at each leaf of $S_1$, $\alpha_\beta$ is dropped on the basis of $\alpha_\beta \implc{0} \alpha$ or $\alpha_\beta \implc{d} \beta$; similarly for $S_2$. We modify the proof forest as follows.

{\it Case 1:} At all leaves of $S_1$, case (A) is applied on the basis of $\alpha_\beta \implc{0} \alpha$. Then we remove the subtree rooted at $L_2$. Furthermore, in each node of $S_1$, we replace $\alpha_\beta$ by $\lambda^\alpha_1, \ldots, \lambda^\alpha_n$. As a consequence, some nodes $L'$ have a child $L'' \subseteq L'$; this is the case where, in the original proof forest, case (A) applies on the basis of $\alpha \implc{0} \alpha_\beta \lor \alpha_\gamma$ or $\alpha_\beta \implc{0} \alpha$. We then remove the descendants of $L'$ and connect the children of $L''$ directly to $L'$.

{\it Case 2:} At all leaves of $S_2$, case (A) is applied on the basis of $\alpha_\gamma \implc{0} \alpha$. Then we proceed similarly as in Case 1.

{\it Case 3:} There is a leaf $L_1' = \{\alpha_\beta\}$ of $S_1$ at which case (B) is applied on the basis of $\alpha_\beta \implc{d} \beta$, and a leaf $L_2' = \{\alpha_\gamma\}$ of $S_2$ at which case (B) is applied on the basis of $\alpha_\gamma \implc{d} \gamma$. We then remove all descendants from $L$ and connect to $L$ instead the subtrees rooted at the children of $L_1'$ and $L_2'$ with weight $d$. Then condition (B) applies to $L$, the used implication being $\alpha \implc{d} \beta \lor \gamma$.

We finally apply Lemma \ref{lem:no-unused-literal} to ensure that there are no unused literals, and we repeat the same procedure as long as the variables $\alpha_\beta$ and $\alpha_\gamma$ occur.
\end{proof}

Lemma \ref{lem:extension-single} was only a preparatory step for the subsequent Lemma, which contains the key argument for the proof of the completeness theorem for \LAE. Namely, we show that a theory $\mathcal T$ can be extended to a theory $\bar{\mathcal T}$ such that condition (E1) applies for all implications of the form $\alpha \implc{d} \beta \lor \gamma$ provable from $\bar{\mathcal T}$, not only for those provable in the original theory $\mathcal T$.

\begin{lemma} \label{lem:extension-whole}
Let $\mathcal T$ be a basic theory. Then there is a basic theory $\bar{\mathcal T}$ including $\mathcal T$ with the following property:
\AxiomeAnfang

\Axiom{F1} Let $\alpha, \beta, \gamma$ be Boolean formulas built up from variables that occur in $\bar{\mathcal T}$. If $\bar{\mathcal T} \proves \alpha \implc{d} \beta \lor \gamma$ for some $d \in \Reals^+$, then there are Boolean formulas $\alpha_\beta, \alpha_\gamma$ such that $\bar{\mathcal T}$ proves the following:
\begin{equation} \label{fml:disjunction-noch}
\alpha_\beta \implc{d} \beta, \quad \alpha_\gamma \implc{d} \gamma, \quad
\alpha_\beta \implc{0} \alpha, \quad \alpha_\gamma \implc{0} \alpha, \quad
\alpha \implc{0} \alpha_\beta \lor \alpha_\gamma.
\end{equation}

\Axiom{F2} Let $\zeta$ and $\eta$ be Boolean formulas built up from variables that occur in $\mathcal T$. If $\bar{\mathcal T} \proves \zeta \implc{r} \eta$, then $\mathcal T \proves \zeta \implc{r} \eta$.
\Axiom{F3} Every weight of an implication in $\bar {\mathcal T}$ is a finite sum of weights of implications in $\mathcal T$.
\AxiomeEnde
\end{lemma}

\begin{proof}
We put ${\mathcal T}_0 = {\mathcal T}$, and for each $i \geq 0$, we let ${\mathcal T}_{i+1} = {\mathcal T}_i'$, the extension of ${\mathcal T}_i$ according to Lemma \ref{lem:extension-single}. We put $\bar{\mathcal T} = \bigcup_i {\mathcal T}_i$.

Let $\Phi$ be an implication containing only variables occurring in $\bar{\mathcal T}$. If $\bar{\mathcal T} \proves \Phi$, there is an $i$ such that $\Phi$ contains only variables occurring in ${\mathcal T}_i$ and ${\mathcal T}_i$ proves $\Phi$. Thus (F1) holds by condition (E1) of Lemma \ref{lem:extension-single}. Furthermore, assume $\bar{\mathcal T} \proves \zeta \implc{r} \eta$, where the variables occurring in $\zeta$ and $\eta$ all occur in $\mathcal T$. Then ${\mathcal T}_i \proves \zeta \implc{r} \eta$ for some $i$. By condition (E2) of Lemma \ref{lem:extension-single}, if $i \geq 1$, ${\mathcal T}_{i-1} \proves \zeta \implc{r} \eta$, and arguing subsequently in the same way we conclude that in fact $\mathcal T$ proves $\zeta \implc{r} \eta$, and (F2) follows. An inductive argument also shows (F3).
\end{proof}

The extension of a basic theory specified by Lemma \ref{lem:extension-whole} is easily seen to fulfil Lemma \ref{lem:smallest-weight} as well.

\begin{lemma} \label{lem:smallest-weight-whole}
Let $\mathcal T$ be a finite theory and let $\bar{\mathcal T}$ be the extension of $\mathcal T$ according to Lemma {\rm \ref{lem:extension-whole}}. Let $\zeta, \eta$ be Boolean formulas such that $\bar{\mathcal T} \proves \zeta \implc{s} \eta$ for some $s \in \Reals^+$. Then there is an $r \in \Reals^+$ with the following properties: {\rm (i)} $r$ is the smallest value such that $\bar{\mathcal T} \proves \zeta \implc{r} \eta$ and {\rm (ii)} $r$ is a finite sum of weights of implications in $\mathcal T$.
\end{lemma}

\begin{proof}
By condition (F3) of Lemma \ref{lem:extension-whole}, each weight of an implication contained in $\bar{\mathcal T}$ is a sum of weights of implications contained in the finite theory $\mathcal T$. Hence we can argue as in the proof of Lemma \ref{lem:smallest-weight}.
\end{proof}

We finally arrive at our main result.

\begin{theorem} \label{thm:completeness-LAE}
Let $\mathcal T$ be a finite theory and $\Phi$ an implication. Then $\mathcal T$ proves $\Phi$ in \pLAE{} if and only if $\mathcal T$ semantically entails $\Phi$ in \LAE.
\end{theorem}

\begin{proof}
The ``only if'' part holds by Proposition \ref{prop:soundness}.

To see the ``if'' part, assume that $\mathcal T$ does not prove the implication $\zeta \implc{r} \eta$. Let $\bar{\mathcal T}$ the extension of $\mathcal T$ according to Lemma \ref{lem:extension-whole}. By (F2), then also $\bar{\mathcal T}$ does not prove $\zeta \implc{r} \eta$.

For Boolean formulas $\alpha$ and $\beta$ built up from variables occurring in $\bar{\mathcal T}$, we put now $\alpha \Tleq \beta$ if $\bar{\mathcal T} \proves \alpha \implc{0} \beta$, and we put $\alpha \Tequ \beta$ if $\alpha \Tleq \beta$ and $\beta \Tleq \alpha$. Then $\Tequ$ is an equivalence relation inducing a Boolean algebra $\mathcal B$. Note that $\mathcal B$ is separable because the number of variables is countable.

We denote the $\Tequ$-class of a Boolean formula $\alpha$ by $\Tequcl{\alpha}$. For each pair $\Tequcl{\alpha}, \Tequcl{\beta} \in {\mathcal B}$, we put
\begin{equation} \label{fml:quasimetric-on-BA}
p(\Tequcl{\alpha}, \Tequcl{\beta}) \;=\; \min\; 
  \{ t \in \ExtReals \colon \bar{\mathcal T} \proves \alpha \implc{t} \beta \},
\end{equation}
the existence of the minimum being ensured by Lemma \ref{lem:smallest-weight-whole}. We shall write in the sequel simply ``$p(\alpha, \beta)$'' instead of ``$p(\Tequcl{\alpha}, \Tequcl{\beta})$''. For any pair of Boolean formulas $\alpha, \beta$ and $d \in \Reals^+$, we have by (\ref{fml:quasimetric-on-BA})
\begin{equation} \label{fml:d}
\bar{\mathcal T} \proves \alpha \implc{d} \beta
\quad\text{if and only if}\quad
p(\alpha, \beta) \leq d.
\end{equation}

We have to show that $p$ fulfils the conditions (QB1)--(QB6). By Lemmas \ref{thm:metricBA-representation} and \ref{lem:quasimetric-space-2}(ii), it will then follow that there is a model such that all elements of $\bar{\mathcal T}$, thus in particular all elements of $\mathcal T$, are satisfied, but $\zeta \implc{r} \eta$ is not.

(QB1): By (\ref{fml:d}), we have $p(\alpha, \beta) = 0$ iff $\bar{\mathcal T} \proves \alpha \implc{0} \beta$ iff $\Tequcl{\alpha} \leq \Tequcl{\beta}$.

(QB2): For any $d \in \Reals^+$, we have $p(\alpha \lor \beta, \gamma) \leq d$ iff $\bar{\mathcal T} \proves \alpha \lor \beta \implc{d} \gamma$ iff $\bar{\mathcal T} \proves \alpha \implc{d} \gamma$ and $\bar{\mathcal T} \proves \beta \implc{d} \gamma$ iff $p(\alpha, \beta) \vee p(\alpha,\gamma) \leq d$.

(QB3): Assume that $\Tequcl{\gamma} \leq \Tequcl{\beta}$. Then $\bar{\mathcal T} \proves \gamma \implc{0} \beta$ and hence $p(\alpha, \beta) \leq p(\alpha, \gamma)$.

(QB4): Let $p(\alpha, \beta) \leq d$ and $p(\beta, \gamma) \leq e$. Then $\bar{\mathcal T} \implc{d+e} \gamma$ and thus $p(\alpha, \gamma) \leq d + e$. It follows $p(\alpha, \gamma) \leq p(\alpha, \beta) + p(\beta, \gamma)$.

(QB5): Assume that $p(\alpha, \beta \lor \gamma) \leq d$. Then $\bar{\mathcal T} \proves \alpha \implc{d} \beta \lor \gamma$, and by condition (F1) of Lemma \ref{lem:extension-whole}, there are Boolean formulas $\alpha_\beta, \alpha_\gamma$ such that $\alpha \Tequ \alpha_\beta \lor \alpha_\gamma$ and $\bar{\mathcal T} \proves \alpha_\beta \implc{d} \beta, \; \alpha_\gamma \implc{d} \gamma$, that is, $p(\alpha_\beta, \beta), \; p(\alpha_\gamma, \gamma) \leq d$. We conclude that $\Tequcl{\alpha_\beta}, \Tequcl{\alpha_\gamma} \in {\mathcal B}$ are such that $\Tequcl{\alpha} = \Tequcl{\alpha_\beta} \vee \Tequcl{\alpha_\gamma}$ and $p(\alpha_\beta, \beta) \vee p(\alpha_\gamma, \gamma) \leq p(\alpha, \beta \lor \gamma)$.

To complete the proof of (QB5), we have to show that the last inequality is in fact an equality. Let $e = p(\alpha_\beta, \beta) \vee p(\alpha_\gamma, \gamma)$; then $p(\alpha_\beta, \beta) \leq e$ and $p(\alpha_\gamma, \gamma) \leq e$ implies $\bar{\mathcal T} \proves \alpha_\beta \implc{e} \beta$ and $\bar{\mathcal T} \proves \alpha_\gamma \implc{e} \gamma$, hence $\bar{\mathcal T} \proves \alpha_\beta \lor \alpha_\gamma \implc{e} \beta \lor \gamma$ and $\bar{\mathcal T} \proves \alpha \implc{e} \beta \lor \gamma$, that is, $p(\alpha, \beta \lor \gamma) \leq e$.

(QB6): Let $\Tequcl{\alpha_0}, \Tequcl{\alpha_1}, \ldots, \Tequcl{\beta} \in {\mathcal B}$ be such that $\inf_i p(\alpha_i, \beta) = 0$. By Lemma \ref{lem:smallest-weight-whole}, the smallest non-zero element in the image of $p$ is the smallest among the non-zero weights of the implications in $\mathcal T$. Since $\mathcal T$ is finite, we conclude that there is a $j$ such that $p(\alpha_j, \beta) = 0$. By (QB1), we have $\Tequcl{\alpha_j} \leq \Tequcl{\beta}$.
\end{proof}

A final note concerns the notion of a min-space, which was mentioned in the Section \ref{sec:LAE}. The countermodel constructed in the proof of Theorem \ref{thm:completeness-LAE} is in fact a min-space. The quasimetric $p$, defined by (\ref{fml:quasimetric-on-BA}), has by Lemma \ref {lem:smallest-weight-whole} the property that, given any $u > 0$, the image of $p$ contains only finitely many values smaller than $u$. The image of the quasimetric $q$, defined by (\ref{fml:q-1}) and (\ref{fml:q-2}), is consequently contained in the image of $p$; hence $q$ has the same property. We conclude that Theorem \ref{thm:completeness-LAE} still holds if we modify the definition of \LAE{} restricting the models to min-spaces.

\section{Conclusion}
\label{sec:Conclusion}

The logic of approximate entailment, which goes back to E.\ Ruspini \cite{Rus}, is based on an idea that is as simple as convincing. The aim is to make precise what it means for a property $\alpha$ to imply another property $\beta$ approximately. Assume that $\alpha$ and $\beta$ correspond to subsets $A$ and $B$, respectively, of a set of worlds $W$. Following \cite{Rus}, all what we need to add is a {\it distance function} on $W$, that is, a map $d$ assigning a value to pairs of worlds that reasonably measures their distinctness. Then $\alpha$ is considered to imply $\beta$ {\it to the degree of imprecision} $d$ if $A$ is a subset of the $d$-neighbourhood of $B$.

To associate with this idea a logic on semantic grounds is straightforward; to axiomatise this logic, however, turns out to be difficult. Those few sound rules that are listed in Definition \ref{def:pLAE} are easily found; to continue on this basis is a challenge. Our question has been if this small set of rules is not already sufficient to axiomatise a logic of the mentioned kind. We have given an affirmative answer for the case where the models are quasimetric spaces.

We may certainly say that the road to this goal was stony. It was necessary to develop the basis of a proof theory for the proposed calculus and to prove quite an amount of auxiliary lemmas. Future efforts in this area might well aim at different techniques, so that the proof of our completeness theorem could be shortened.

The present work gives rise to a considerable amount of open problems. First to mention, our completeness theorem covers finite theories; the case of infinite theories is open. We note that we may hardly proceed along the same lines as we did here; our procedure heavily relies on the finiteness assumption.

Second, we have introduced in \cite{EGRV} a counterpart to the logic considered here, called the Logic of Strong Entailment, or \LSE{} for short. Whereas in \LAE{} implicational relationships holding only approximately are considered, \LSE{} not only requires implications to hold exactly, but even to be invariant under quantified changes. We wonder if a proof calculus along the present lines could be found for \LSE{} as well.

Another question might not lead to immediate results, but should be explored nonetheless. The models constructed in the present context are, from the point of view of applications, somewhat unnatural; the image of the distance function contains only finitely many values below any given $m \in \Reals^+$. We ask if there is a logic of approximate entailment that corresponds more closely to the common situation that the set of world is (a subset of) $\Reals^n$, endowed with the Euclidean metric.

Speaking about the Euclidean metric, we are finally led to the probably most significant open problem in the present context. What we have in mind is the symmetry of the distance function. Does our completeness theorem still hold if we restrict to metric spaces, that is, those spaces that are based on a symmetric quasimetric? We conjecture that the answer is positive. We guess that proof-theoretical methods, however, are not suitable to deal with this case. We propose to examine instead the question if quasimetric spaces can be suitably embedded in metric spaces, such that the quasimetric of the former corresponds to the Hausdorff quasimetric of the latter space. A result of this kind is contained in \cite{Vit}; unfortunately, the proposed representation seems not to be applicable in the present context and an independent approach needs to be found.

\subsubsection*{Acknowledgements} 
The author acknowledges the support of the Austrian Science Fund (FWF): project I 1923-N25 (New perspectives on residuated posets).

He would moreover like to thank the anonymous reviewers for their constructive criticism, which led to an improvement of this paper.

\end{document}